\documentclass[10pt]{amsart}

\usepackage{amsmath,amsthm,amsfonts,latexsym,epsfig}

\newtheorem{thm}{Theorem}[section]

\newtheorem{pro}[thm]{Proposition}

\newtheorem{cor}[thm]{Corollary}

\newtheorem{rmk}[thm]{Remark}

\newcommand{\al}{\alpha}
\newcommand{\be}{\beta}
\newcommand{\ga}{\gamma}

\newcommand{\la}{{\lambda}}

\newcommand{\bbR}{\mathbb{R}}
\newcommand{\del}{\partial}

\newcommand{\Tr}{\mbox{Tr }}
\newcommand{\Lt}{L^{(3)}}
\newcommand{\Pf}{\mbox{Pf }}
\newcommand{\sgn}{\mbox{sgn }}

\newcommand{\dmrjdel}[1]{}

\title{The Non-Orientable Map Asymptotics Constant $p_g$}
\author[S.~R.~Carrell]{S.~R.~Carrell$^*$}

\thanks{
${\hspace{-1ex}}^*$Department of Combinatorics and Optimization,
                                                University of Waterloo, Waterloo, Ontario, Canada.;    \\
${\hspace{.35cm}}$ \texttt{srcarrell@uwaterloo.ca}}

\begin{document}

\begin{abstract}
Using the pfaffian structure of the generating series for locally orientable maps, we show that the generating series satsifies a nonlinear differential
equation called the BKP equation. Using this we are able to derive a cubic differential equation which is satisfied by the generating series for
locally orientable triangulations. As a result, we prove a conjecture of Garoufalidis and Mari\~no concerning the constant $p_g$ which appears in
asymptotic formulas for a variety of rooted maps on non-orientable surfaces. This allows one to determine the asymptotic expansion for $p_g$ up to
an unknown Stokes constant.
\end{abstract}

\maketitle

\section{Introduction}

A map is a connected graph embedded in a compact connected surface in such a way that the regions delimited by the graph, called faces, are homeomorphic to open discs.
Loops and multiple edges are allowed. A rooted map is one in which an angular sector incident to a vertex is distinguished, and the latter is called the root vertex.
The root edge is the edge encountered when traversing the distinguished angular sector clockwise around the root vertex. We say that a map is orientable, non-orientable
or locally orientable if the same is true of the underlying surface.

In \cite{BC86}, Bender and Canfield studied the number $T_g(n)$ of $n$-edged rooted maps on an orientable surface of type $g$ and the number $P_g(n)$ of $n$-edged rooted
maps on a non-orientable surface of type $g$. They showed that asymptotically the numbers behaved as
	\begin{align*}
		T_g(n)	&\sim t_g n^{5(g-1)/2} 12^n, \\
		P_g(n)  &\sim p_g n^{5(g-1)/2} 12^n, \mbox{ when } g > 0,
	\end{align*}
where $t_g$ and $p_g$ are constants which depend only on $g$. Unfortunately, determining the constants $t_g$ and $p_g$ proved to be very difficult. Later, Gao\cite{G93} 
showed that if $C$ denotes a class of rooted maps (for example, 2-connected, triangulations and $2d$-regular) and $M_g(C,n)$ is the number of maps in class $C$ which
are of type $g$ and which have $n$ edges then for many such classes,
	\[ M_g(C,n) \sim \al t_g (\be n)^{5(g-1)/2}\ga^n, \]
if the maps are orientable and
	\[ M_g(C,n) \sim \al p_g (\be n)^{5(g-1)/2}\ga^n, \]
if the maps are non-orientable. Here $\al, \be$ and $\ga$ depend on the class of maps considered.

Goulden and Jackson\cite{GJ08} showed that the number of rooted orientable triangulations satisfies a quadratic recurrence equation. Equivalently, this
implies that the generating series for rooted orientable triangulations satisfies a quadratic differential equation. Goulden and Jackson proved this using the
fact that the generating series for rooted maps with respect to vertex degrees satisfies a family of differential equations known as the KP hierarchy,
essentially coming from the fact that the generating series in question has a determinantal structure (the Schur function expansion has coefficients which
satisfy the Pl\"ucker relations). Using this quadratic differential equation and the fact that the class of triangulations has the asymptotic behaviour described
by Gao, it was shown by Bender, Richmond and Gao\cite{BRG08} that the map asymptotics constants $t_g$ satisfied a quadratic recurrence. It was then shown by
Garoufalidis, L\^e and Mari\~no\cite{GLM08} that the scaled generating series for the map asymptotics constant $t_g$ satisfied a quadratic
differential equation equivalent to the Painlev\'e I equation. In particular they showed that if $u_g = - 2^{g-2} \Gamma\left( \frac{5g-1}{2} \right) t_g$ then the series
	\[ u(z) = z^{1/2} \sum_{g=0}^\infty u_g z^{-5g/2}, \]
satisfies
	\[ u^2 - \frac{1}{6}u'' = z. \] 

By analogy to the orientable case and motivated by some results in mathematical physics concerning integrals over real symmetric matrices, Garoufalidis and Mari\~no\cite{GM10} 
conjecture that similar to the series $u(z)$ above, the scaled generating series for the non-orientable map asymptotics constants also satisfies a simple differential
equation. The main result of this paper is the following which appears as Conjecture 1 in Garoufalidis and Mari\~no\cite{GM10}.
\begin{thm}\label{p_gTheorem}
	Let
		\[ v_g = 2^{\frac{g-3}{2}} \Gamma\left( \frac{5g-1}{4} \right) p_{\frac{g+1}{2}}. \]
	The series
		\[ v(z) = z^{1/4}\sum_{g=0}^\infty v_g z^{-5g/4} \]
	satisfies the differential equation
		\[ 2v' - v^2 + 3u = 0, \]
	where the series $u$ is the same as the series $u$ above.
\end{thm}

One of the main advantages to the differential equations satisfied by the generating series for the $t_g$ and $p_g$ constants is that they can be used to determine
the corresponding asymptotic behaviour. The differential equation for $u(z)$ was used by Garoufalidis, L\^e and Mari\~no\cite{GLM08} to determine the asymptotic
behaviour of $t_g$ to all orders. In particular, the following Theorem is proven in Appendix A of \cite{GLM08}.
\begin{thm}\label{t_gTheorem}
Suppose the series
	\[ u(z) = z^{1/2} \sum_{g=0}^\infty u_g z^{-5g/2} \]
satisfies
	\[ u^2 - \frac{1}{6}u'' = z. \]
Then as $g \to \infty$,
	\[ u_g \sim A^{-2g + \frac{1}{2}} \Gamma\left( 2g - \frac{1}{2} \right) \frac{S}{2\pi i} \left\{ 1 + \sum_{\ell = 1}^\infty \frac{\mu_\ell A^\ell}{\prod_{m=1}^\ell (2g - 1/2 - m)} \right\}, \]
where $A = \frac{8\sqrt{3}}{5}, S = -i\frac{3^{\frac{1}{4}}}{\sqrt{\pi}}$ and the $\mu_\ell$ are defined by the recursion relation
	\[ \mu_\ell = \frac{5}{16 \sqrt{3}\ell} \left\{ \frac{192}{25} \sum_{k=0}^{\ell - 1} \mu_k u_{(\ell - k + 1)/2} - \left( \ell - \frac{9}{10} \right)\left( \ell - \frac{1}{10} \right)\mu_{\ell-1} \right\}, \qquad \mu_0 = 1, \]
with the convention that $u_{n/2} = 0$ if $n$ is odd.
\end{thm}

Similarly, the differential equation in Theorem~\ref{p_gTheorem} was used by Garoufalidis and Mari\~no\cite{GM10} to determine the complete asymptotic behaviour of $v_g$ and hence,
by Theorem~\ref{p_gTheorem}, $p_g$ up to an unknown constant $S'$. The complete details can be found in Garoufalidis and Mari\~no\cite{GM10} and the following
appears as Theorem 1 in \cite{GM10}.
\begin{thm}\label{p_gAsymptotics}
Suppose the series
	\[ v(z) = z^{1/4}\sum_{g=0}^\infty v_g z^{-5g/4} \]
satisfies the differential equation
	\[ 2v' - v^2 + 3u = 0, \]
where $u(z)$ is the series described in Theorem~\ref{t_gTheorem}. Then the sequence $v_g$ has an asymptotic expansion of the form
	\[ v_g \sim (A/2)^{-g} \Gamma(g) \frac{S'}{2 \pi i} \left\{ 1 + \sum_{\ell = 1}^\infty \frac{\nu_\ell(A/2)^\ell}{\prod_{m=1}^\ell (g-m)} \right\}, \]
where $A$ is given in Theorem~\ref{t_gTheorem}, $S' \not = 0$ is some non-zero Stokes constant, and the sequence $\nu_\ell$ is defined by the recursion relation
	\[ \nu_\ell = -\frac{4}{5\ell} \sum_{k=0}^{\ell-1} v_{\ell + 1 - k}\nu_k, \qquad \nu_0 = 1. \]
\end{thm}

The rest of this paper is organized as follows. In Section 2 we discuss how the generating series for locally orientable maps (a combination of both orientable and
non-orientable) has a pfaffian structure. This is in analogy with the generating series for orientable maps which has a determinantal structure. As a result we
show that the locally orientable map series satisfies the BKP equation, a pfaffian analog of the KP equation. In addition, we discuss some linear differential
equations which are satisfied by the locally orientable map series which follow by considering the removal of vertices with degree one or two. In Section 3 we
use the differential equations described in Section 2 to derive a cubic differential equation for the specialization of the locally orientable map series to
triangulations. In Section 4 we give a structural result for $\Lt_g(z)$, the generating series for locally orientable triangulations. In particular, we show that
it can be written as a rational series in terms of an auxilliary algebraic series. This type of structure seems to arise frequently in map enumeration and in
permutation factorization problems as in, for example, Goulden, Guay-Paquet and Novak's work on monotone Hurwtiz numbers\cite{GGN13}, however the reason for this
is not clear. Lastly, in Section 5, we combine the structure theorem and the cubic differential equation to derive some results about the asymptotic behaviour
of locally orientable triangulations. As an application, we are able to prove Garoufalidis and Mari\~no's conjecture about the non-orientable map asymptotics constants. 

\section{Symmetric Matrix Integrals and Locally Orientable Maps}

We define an averaging (expectation) operator over $\bbR^N$ as follows:
	\[ \langle f(\lambda) \rangle_N = \frac{1}{N!} \int_{\bbR^N} |V(\la)| f(\lambda) \exp \left( \frac{-p_2(\la)}{4} \right) d\lambda. \]
Here $V(\lambda)$ is the Vandermonde determinant and $p_k(\lambda)$ is the $k$th power sum symmetric function. They are given by
	\[ V(\lambda) = \prod_{1 \leq i < j \leq N} (\la_i - \la_j) \qquad \mbox{ and } \qquad p_k(\lambda) = \sum_{1 \leq i \leq N} \la_i^k. \]
This averaging operator is related to a similar operator over the vector space $W_N$ of all $N \times N$ real symmetric matrices $M \in W_N$ with measure
$e^{-\Tr M^2/4}$. In fact, up to a multiplicative constant,
	\[ \langle f(\lambda) \rangle_{N} = \int_{W_N} f(M) e^{-\Tr M^2/4} dM. \]
This relationship follows from the fact that the integrand is conjugation invariant and so we may use the polar decomposition for real symmetric matrices
to reduce the integral to one over $\bbR^N$.

\begin{thm}[Kakei\cite{K99}, Van de Leur\cite{VdL01}, Adler and Moerbeke\cite{AM01}] \label{BKP}
	Suppose
		\[ Z_N = \left\langle \exp\left( \sum_{k \geq 1} \frac{p_k(\lambda)}{2k} t_k \right) \right\rangle_N. \]
	Then
		\[ (\del_1^4 + 3\del_2^2 + 3\del_3\del_1) \log Z_N + 6(\del_1^2 \log Z_N)^2 = 3 \frac{Z_{N+2} Z_{N-2}}{Z_N^2}. \]
\end{thm}
\begin{proof}
	The differential equation which $Z_N$ satisfies is known as the BKP equation and is one of many in the BKP hierarchy. Adler and Moerbeke\cite{AM01} showed that for
	even $N$ the $Z_N$ satisfies each of the differential equations in the BKP hierarchy. Van de Leur\cite{VdL01} then generalized this result to any $N$. In both cases
	the authors used the Fock space approach to integrable hierarchies. For our purposes it will suffice to show that for even $N$,  $Z_N$ satisfies the BKP equation.
	In this case there is a direct proof as shown by Kakei\cite{K99} and which can also be found in Hirota\cite{H04}.
	
	For the remainder of this proof we will work with $2N$ rather than $N$. Note that this is sufficient for our purposes since the particular integral which
	we are interested in (Theorem~\ref{MtxIntegral}) is a series with coefficients which are polynomial in $N$ (see Remark~\ref{polynomiality}).
	Recall that given a $2N \times 2N$ antisymmetric matrix $A_{2N} = [a_{i,j}]_{1 \leq i,j \leq 2N}$ we may define the pfaffian as
		\[ \Pf(A_{2N}) = \sum \sgn\left( \begin{array}{cccc} 1 & 2 & \cdots & 2N \\ j_1 & j_2 & \cdots & j_{2N} \end{array} \right) a_{j_1,j_2} a_{j_3,j_4} \cdots a_{j_{2N-1},j_{2N}}, \]
	where the summation is over all $(j_1, \cdots j_{2N})$ such that $j_1 < j_3 < \cdots < j_{2N-1}$ and $j_1 < j_2, \cdots, j_{2N-1} < j_{2N}$. For our purposes
	it will be more convenient to write the pfaffian in a different way. Given a pairing $(i,j)$ we define a pfaffian corresponding to a sequence $(a_1, a_2, \cdots, a_{2m})$ as
		\[ (a_1, a_2, \cdots, a_{2m}) = \sum_{j=2}^{2m} (-1)^j (a_1, a_j)(a_2, a_3, \cdots, \widehat{a_j}, \cdots, a_{2m}), \]
	where $\widehat{a_j}$ means that $a_j$ is omitted.
	
	Using a Theorem of de Bruijn\cite{dB55} (this can also be found in Mehta\cite{M04}), we may write $Z_{2N}$ as a pfaffian. In particular,
		\[ Z_{2N} = \Pf \left[ \mu_{i,j}(t) \right]_{1 \leq i,j \leq 2N}, \]
	where
		\[ \mu_{i,j}(t) = \mathop{\int \int}_{x < y} \left( x^{i-1} y^{j-1} - y^{i-1} x^{j-1} \right) \exp \left( \frac{-x^2 - y^2}{4} + \sum_{k \geq 1} \frac{x^k + y^k}{2k} t_k \right). \]
	Similarly, we may write
		\[ Z_{2N} = (1, 2, \cdots, 2N), \]
	with the pairing given by $(i,j) = \mu_{i,j}(t)$.
	
	We may also check that for any positive integer $k$,
		\[ 2k \frac{\del}{\del t_k} \mu_{i,j}(t) = \mu_{i+k,j}(t) + \mu_{i,j+k}(t). \]
	Using the linear differential equation satisfied by the entries of the pfaffian corresponding to $Z_{2N}$ we may find identities which are satisfied by $Z_{2N}$ itself. In particular,
	it follows that
		\[ 2k \frac{\del}{\del t_k} (a_1, a_2, \cdots, a_{2N}) = \sum_{j = 1}^{2N} (a_1, a_2, \cdots, a_{j+k}, \cdots, a_{2N}). \]
	
	Furthermore, it is straightforward to show that pfaffians satisfy the identities
		\begin{align*}
			(a_1, a_2, \cdots, a_{2m}, & 1, 2, \cdots, 2n)(1, 2, \cdots, 2n) = \\
									   & \sum_{j=2}^{2m} (-1)^j (a_1, a_j, 1, 2, \cdots, 2n)(a_2, a_2, \cdots, \widehat{a_j}, \cdots, a_{2m}, 1, 2, \cdots, 2n),
		\end{align*}
	where $\widehat{a_j}$ means that $a_j$ is omitted. Using this pfaffian identity and the linear differential equations satisfied by $Z_{2N}$ the result follows.
\end{proof}

\begin{thm}[Mehta\cite{M04}, Van de Leur\cite{VdL01}, Adler and Moerbeke\cite{AM01}] \label{Virasoro}
	Suppose
		\[ Z_N = \left\langle \exp\left( \sum_{k \geq 1} \frac{p_k(\lambda)}{2k} t_k \right) \right\rangle_N. \]
	Then
		\[ \del_1 \log Z_N = \sum_{i \geq 1} i t_{i+1} \del_i \log Z_N + \frac{N t_1}{2}, \]
	and
		\[ \del_2 \log Z_N = \sum_{i \geq 1} \frac{i}{2} t_i \del_i \log Z_N + \frac{N(N+1)}{4}. \]
\end{thm}
\begin{proof}
	This can be shown directly using the fact that the integration measure is translation invariant and the details of this method can be found in Mehta\cite{M04}. Alternate proofs using
	Fock space methods can be found in Van de Leur\cite{VdL01} and Adler and Moerbeke\cite{AM01}.
	
	For our purposes, via Theorem~\ref{MtxIntegral}, the series of interest is a generating series for rooted maps. In this case, the two equations correspond to adding / removing
	a vertex of degree one and adding / removing a vertex of degree two.
\end{proof}

Let $\ell_{k,\al}$ be the number of locally orientable maps with $k$ faces and vertex partition given by $\al$. Then the genus series for maps in locally orientable surfaces is defined to be
	\[ L(t;y) = \sum_{k \geq 1} \sum_{\al} \ell_{k,\al} y^k t_\al. \]
\begin{rmk} \label{polynomiality}
	Let \[ \ell_\al(y) = [t_\al] L(t;y) = \sum_{k\geq 1} \ell_{k,\al}y^k. \] Then, since for any fixed map the sum of the vertex degrees is equal to the sum of the face degrees, we must
	have $k \leq |\al|$. Thus, each $\ell_{\al}(y)$ is a polynomial in $y$.
\end{rmk}

\begin{thm}[Goulden and Jackson\cite{GJ97}] \label{MtxIntegral}
	For any positive integer $N$,
		\[ L(t; N) = 2 E \log \left\langle \exp\left( \sum_{k \geq 1} \frac{p_k(\lambda)}{2k} t_k \right) \right\rangle_N, \]
	where
		\[ E = \sum_{k \geq 1} k t_k \del_k. \]
\end{thm}

In particular, Theorem~\ref{MtxIntegral} above implies that the differential equations in Theorem~\ref{BKP} and Theorem~\ref{Virasoro} are also satisfied by the generating
series for locally orientable maps. This collection of differential equations is the primary tool used in the remainder of this paper.

\section{A Cubic Differential Equation for Triangulations}

Let $\Lt(x,y) = \left. L(t;y) \right|_{t_i = x \delta_{i,3}}$. That is, let $\Lt(x,y)$ be the generating series for locally orientable cubic maps. Similarly, let $\Lt(x;w)$ be the generating series
for locally orientable cubic maps with vertices marked by $x$ and Euler characteristic marked by $w$. By Euler's formula,
	\[ \Lt(x;w) = w^2 \left. \Lt(x,y) \right|_{x \mapsto xw^{1/2}, y \mapsto w^{-1}}. \]
We begin by using Theorem~\ref{BKP} and Theorem~\ref{Virasoro} to derive a cubic differential equation for $\Lt(x;w)$.

\begin{thm} \label{MasterEquationGenus}
	Let
	\begin{align*}
		T &= \Lt(x;w) + w + 1 - \frac{1}{2x^2}, \\
		V &= w^{-2}\left( (1+2w)^2 \Lt\left( x(1+2w)^{1/2}, \frac{w}{1+2w} \right) \right. \\
		  & \qquad \left. + (1-2w)^2 \Lt\left(x(1-2w)^{1/2}, \frac{w}{1-2w}\right) - 2\Lt(x;w) \right).
	\end{align*}
	Then
	\begin{align*}
		4x^4 & w^2 (D+12)^2 (D+8)(D+4)T - (D+6)DT + 12x^4(D+12)((D+4)T)^2 \\
					 &= V \left( 2x^4w^2(D+12)(D+8)(D+4)T - \frac{1}{2}(D+6)T + 6x^4((D+4)T)^2 \right).
	\end{align*}
\end{thm}
\begin{proof}
	Recall from Theorem~\ref{Virasoro} that if
		\[ Z_N = \left\langle \exp\left( \sum_{k \geq 1} \frac{p_k(\lambda)}{2k} t_k \right) \right\rangle_N, \]
	then
		\begin{align*}
			\del_1 \log Z_N	&= \sum_{i \geq 1} i t_{i+1} \del_i \log Z_N + \frac{N t_1}{2}, \\
			\del_2 \log Z_N &= \sum_{i \geq 1} \frac{i}{2} t_i \del_i \log Z_N + \frac{N(N+1)}{4},
		\end{align*}
	and (from Theorem~\ref{BKP}) that
		\[ (\del_1^4 + 3\del_2^2 - 3\del_3\del_1)\log Z_N + 6(\del_1^2 \log Z_N)^2 = \frac{3}{4} \frac{Z_{N+2}Z_{N-2}}{Z_N^2}. \]
	Let $Y_N = \left. \log Z_N \right|_{t_i = 0, i > 3}$. Then the equations above imply that
		\begin{align}
			\del_1 Y_N &= t_2 \del_1 Y_N + 2 t_3 \del_2 Y_N + \frac{N t_1}{2}, \label{Vir1} \\
			\del_2 Y_N &= \frac{1}{2} t_1 \del_1 Y_N + t_2 \del_2 Y_N + \frac{3}{2} t_3 \del_3 Y_N + \frac{N(N+1)}{4}, \label{Vir2} \\
			\del_1^4 Y_N &+ 3\del_2^2 Y_N - 3\del_3 \del_1 Y_N + 6(\del_1^2 Y_N)^2 = \frac{3}{4} \exp(Y_{N+2} + Y_{N-2} - 2Y_N). \label{BKP2} 
		\end{align}
	The $\del_2 Y_N$ term can be eliminated from \eqref{Vir1} using \eqref{Vir2} and similarly the $\del_1 Y_N$ term can be eliminated from \eqref{Vir2} using \eqref{Vir1}. Letting $D = 3 t_3 \del_3$ this gives
		\begin{align}
			\del_1 Y_N	&= \frac{t_3}{A(1-t_2)} D Y_N + C, \label{mVir1} \\
			\del_2 Y_N  &= \frac{1}{2A} D Y_N + \frac{B}{A}, \label{mVir2}
		\end{align}
	where
		\begin{align*}
			A	&= 1 - \frac{t_1 t_3}{1 - t_2} - t_2, \\
			B	&= \frac{Nt_1^2}{4(1-t_2)} + \frac{N(N+1)}{4},
		\end{align*}
	and
		\[ C = \frac{2t_3B}{A(1-t_2)} + \frac{Nt_1}{2(1-t_2)}. \]
	Letting $T = \left. 2DY_n \right|_{t_i = x\delta_{i,3}} + N(N+1) - \frac{N}{2x^2} = L^{(3)}(x,N) + N(N+1) - \frac{N}{2x^2}$ (the second equality follows from Theorem~\ref{MtxIntegral}), \eqref{mVir1} and
	\eqref{mVir2} can be used to determine
		\begin{align}
			\left. \del_1^2 Y_N \right|_{t_i = x \delta_{i,3}}	&= \frac{x^2}{2}(D+4)T, \label{Y11} \\
			\left. \del_1^4 Y_N \right|_{t_i = x \delta_{i,3}}	&= \frac{x^4}{2}(D+12)(D+8)(D+4) T, \label{Y1111} \\
			\left. \del_2^2 Y_N \right|_{t_i = x \delta_{i,3}}	&= \frac{1}{8}(D+2)T - \frac{N}{4x^2}, \label{Y22} \\
			\left. \del_1\del_3 Y_N \right|_{t_i = x \delta_{i,3}} &= \frac{1}{6}(D+3)T - \frac{N}{4x^2}, \label{Y13}
		\end{align}
	where now $D = 3x\del_x$. Using the equations above, \eqref{BKP2} becomes
		\begin{align*}
			\frac{x^4}{2} & (D+12) (D+8)(D+4)T + \frac{3}{8}(D+2)T - \frac{3}{4}\frac{N}{x^2} - \frac{1}{2}(D+3)T + \frac{3}{4}\frac{N}{x^2} \\
			                     & + 6\left( \frac{x^2}{2}(D+4)T \right)^2 = \frac{3}{4} \left( \left. Y_{N+2} \right|_{t_i = x\delta_{i,3}} + \left. Y_{N-2} \right|_{t_i = x\delta_{i,3}} - \left. 2Y_N \right|_{t_i = x\delta_{i,3}} \right).
		\end{align*}
	Simplifying, this becomes
		\begin{align*}
			2x^4(D+12) & (D+8)(D+4)T - \frac{1}{2}(D+6)T + 6x^2((D+4)T)^2 \\
			           &= \frac{3}{4} \left( \left. Y_{N+2} \right|_{t_i = x\delta_{i,3}} + \left. Y_{N-2} \right|_{t_i = x\delta_{i,3}} - \left. 2Y_{N} \right|_{t_i = x\delta_{i,3}} \right).
		\end{align*}
	Applying the operator $2D$ to both sides of this equation gives
		\begin{align}
			4x^4 & (D+12)^2(D+8)(D+4)T - (D+6)DT + 12x^4(D+12)((D+4)T)^2 \notag \\
									& = V(2x^4(D+12)(D+8)(D+4)T - \frac{1}{2}(D+6)T + 6x^4((D+4)T)^2), \label{mEquation1}
		\end{align}
	where
		\begin{align*}
			V	&= 2D\left. Y_{N+2} \right|_{t_i = x\delta_{i,3}} + 2D \left. Y_{N-2} \right|_{t_i = x\delta_{i,3}} - 4D \left. Y_{N} \right|_{t_i = x\delta_{i,3}} \\
				&= L^{(3)}(x,N+2) + L^{(3)}(x,N-2) - 2L^{(3)}(x,N).
		\end{align*}
	Now, using the fact that the coefficient of $x^n$ in $L^{(3)}(x,y)$ is a polynomial in $y$ (see Remark~\ref{polynomiality}), we may extract coefficients in \eqref{mEquation1} to get a countable number
	of polynomial identities which are satisfied and thus equation~\eqref{mEquation1} holds with
		\begin{align*}
			T &= \Lt(x,y) + y(y+1) - \frac{y}{2x^2}, \\
			V &= \Lt(x,y+2) + \Lt(x, y-2) - 2\Lt(x,y).
		\end{align*} 
	Recall that
		\[ \Lt(xw^{1/2}, w^{-1}) = w^{-2}\Lt(x;w). \]
	We have
		\begin{align*}
			\Lt(xw^{1/2},w^{-1} + 2)	&= \Lt\left(x(1+2w)^{1/2} \left( \frac{w}{1+2w} \right)^{1/2}, \left( \frac{w}{1+2w} \right)^{-1} \right) \\
										&= w^{-2} (1+2w)^2 \Lt\left( x(1+2w)^{1/2}, \frac{w}{1+2w} \right).
		\end{align*}
	Similarly,
		\[ \Lt(xw^{1/2}, w^{-1} - 2) = w^{-2} (1-2w)^2 \Lt\left( x(1-2w)^{1/2}, \frac{w}{1-2w} \right). \]
	Making the change of variables $x \mapsto w^{1/2}x$, setting $y = w^{-1}$ and then simplifying gives the desired result.
\end{proof}

Since any cubic map must have an even number of vertices, we let $x = z^{1/2}$. Note that in this case $D$ becomes $D = 6z\del_z$. Also, let
	\[ \Lt_g(z) = [w^g]\Lt(z^{1/2};w), \]
and let
	\[ T_g(z) = [w^g]\left(\Lt(z^{1/2};w) + w + 1 - \frac{1}{2z}\right) = \Lt_g(z) + \delta_{g,1} + \left( 1 - \frac{1}{2z} \right)\delta_{g,0}. \]
\begin{cor} \label{MasterEquationGenusRefined}
	For $g \geq 0$,
		\begin{align*}
			4z^2 & (D+12)^2(D+8)(D+4)T_{g-2} - (D+6)DT_g \\
			     & \qquad + 12z^2(D+12)\left( \sum_{i=0}^g \left\{ (D+4)T_i \right\} \left\{ (D+4)T_{g-i} \right\} \right) \\
				 &= \sum_{k=0}^g V_k \left\{ 2z^2(D+12)(D+8)(D+4)T_{g-k-2} \right. \\
				 &  \qquad \qquad \left. - \frac{1}{2}(D+6)T_{g-k} + 6z^2\left( \sum_{i=0}^{g-k} \left\{ (D+4)T_i \right\} \left\{ (D+4)T_{g-k-i} \right\} \right) \right\},
		\end{align*}
	where
		\[ V_k = \sum_{t=0}^k 2^{k-t+2} (1 + (-1)^{k-t}) \sum_{i=0}^{k-t+2} \binom{2-t}{k-t-i+2} \frac{z^i}{i!} \del_z^i \Lt_t(z). \]
\end{cor}
\begin{proof}
	This follows after making the substitution $x = z^{1/2}$ in Theorem~\ref{MasterEquationGenus} and extracting the coefficient of $w^g$.
\end{proof}

\section{A Structure Theorem for $\Lt_g(z)$}

In \cite{G91} Gao studied the generating series for rooted triangulations on arbitrary surfaces. Using a Tutte type equation for the generating series he was able to find
compact representations of the generating series in the projective plane, sphere and torus cases. In each case the generating series is rational in the auxilliary, algebraic
series $s = s(x)$ which is the unique power series solution to $z = \frac{1}{2}s(1-s)(1-2s)$ with $s(0) = 0$. In particular, Gao proves Theorem~\ref{GaoTheorem} below.
Using this same method we show in Theorem~\ref{AlgebraicSeriesTheorem} that the generating series for a surface of any type is always rational when expressed in terms of
the series $s$. Note that the method used below is essentially the same as the method used in \cite{G91} and also in \cite{BC86}. The difference is that here we consider
cubic maps enumerated with respect to the number of vertices and in \cite{G91} Gao consideres triangulations with respect to the number of vertices.

To simplify some of the expression below we let $\eta = 1 - 6s + 6s^2$.
  
\begin{thm}[Gao\cite{G91}] \label{GaoTheorem}
	The generating series $\Lt_0(z)$ and $\Lt_1(z)$ are given in terms of $s$ by
		\begin{align*}
			\Lt_0(z)	&= \frac{2s(1-4s+2s^2)}{(1-s)(1-2s)^2}, \\
			\Lt_1(z)	&= \frac{(1-2s)(1-s+s^2) - (1-6s+6s^2)^\frac{1}{2}}{s(1-s)(1-2s)}.
		\end{align*}
\end{thm}

\begin{thm}\label{AlgebraicSeriesTheorem}
	For $g \geq 0$ the series $\Lt_g$ is a polynomial in $\eta^{\frac{1}{2}}$ with coefficients which are rational series in $s$.
\end{thm}
\begin{proof}
	Let $t_g(n, k, u_1, u_2, \cdots)$ be the number of locally orientable maps with Euler characteristic $g$ which have $n$ vertices, whose root face has degree $k$,
	which have a finite number of distinguished faces the $i$th of which has degree $u_i$ and where every other face has degree three. Let $I = \{i_1 < i_2 < \cdots \}$ be
	a finite set and let
		\[ T_g(x, y, I) = \sum t_g(n, k, u_1, u_2, \cdots) x^n y^k z_{i_1}^{u_1} \cdots. \]
	The series of near-triangular maps ($T_g$) was studied by Gao\cite{G91} using a Tutte type recursion. Here we repeat part of the argument in order to prove the Theorem
	but we refer to the paper for more details. Recall the Tutte type recursion, letting $I$ be a finite set and $w \not \in I$ an integer, then
		\begin{align}
			T_g(x,y,I)	&= y^2 \sum_{j=0/2}^{g} \sum_{S \subseteq I} T_j(x,y,S)T_{g-j}(x,y,I-S) \notag \\
						& \qquad + 2y^3 \left. \frac{\del}{\del z_w} T_{g-1}(x,y,I+\{w\}) \right|_{z_w = y} \notag \\
						& \qquad + y^2 \frac{\del}{\del y} \left( yT_{g-\frac{1}{2}}(x,y,I) \right) \label{ACT.MainEquation} \\
						& \qquad + y^{-1} \left[ T_g(x,y,I) - \delta_{0,g} \delta_{\emptyset, I} x - yT^1_g(x,I) \right] \notag \\
						& \qquad + \sum_{i \in I} \frac{y z_i}{z_i - y} \left[ z_i T_g(x,z_i,I-\{i\}) - yT_g(x,y,I-\{i\}) \right] \notag \\
						& \qquad + \delta_{0,g} \delta_{\emptyset, I} x, \notag
		\end{align} 
	where here the summation index beginning at $0/2$ means to sum over half integers from $0$ to $g$ and where $T^1_g(x,I) = [y]T_g(x,y,I)$. Note also that if we let
	$T^3_g(x,I) = [y^3] T_g(x,y,I)$ then
		\[ T_0^3(x,\emptyset) = T_0^1(x, \emptyset) - x^2, \qquad T_{\frac{1}{2}}^3(x,\emptyset) = T_{\frac{1}{2}}^1(x,\emptyset) - x, \]
	and
		\[ T_g^3(x,\emptyset) = T_g^1(x,\emptyset) \mbox{ for } g \geq 1. \]
	
	Let
		\begin{align*}
			A(x,y)	&= 2y^3 T_0(x,y,\emptyset) + 1 - y, \\
			B(x,y)  &= (1-y)^2 + 4y^3(x - xy + yT_0^1(x,\emptyset)).
		\end{align*}
	Then \eqref{ACT.MainEquation} when $(g,I) = (0,\emptyset)$ is equivalent to
		\[ A(x,y)^2 = B(x,y), \]
	and for $(g,I) \not = (0, \emptyset)$ \eqref{ACT.MainEquation} is equivalent to
		\begin{align*}
			-A(x,y)T_g(x,y,I)	&= y^3 \mathop{\sum_{j=0/2}^{g} \sum_{S \subseteq I}}_{(j,S) \not = (0, \emptyset), (g, I)} T_j(x,y,S)T_{g-j}(x,y,I-S) \\
								& \qquad + 2y^4 \left. \frac{\del}{\del z_w} T_{g-1}(x,y,I+\{w\}) \right|_{z_w = y} \\
								& \qquad + y^3 \frac{\del}{\del y} \left( yT_{g-\frac{1}{2}}(x,y,I) \right) \\
								& \qquad + y^2 \sum_{i \in I} \frac{z_i}{z_i - y} \left[ z_i T_g(x,z_i,I-\{i\}) - yT_g(x,y,I-\{i\}) \right] \\
								& \qquad -yT_g^1(x,I),
		\end{align*} 
	as in \cite{G91}. Let $f(x)$ be a series such that $A(x,f(x)) = 0$. Then, if we define $F^{(k)} = \left. \left( \frac{\del^k F}{\del y^k} \right) \right|_{y = f}$ we see that
		\begin{align*}
			B^{(0)} = (1-f)^2 + 4f^3(x-xf+fT_0^1(x,\emptyset)) = 0, \\
			B^{(1)} = -2(1-f) + 4f^2(3x - 4xf + 4fT_0^1(x,\emptyset)) = 0.
		\end{align*}
	Solving this and setting $f = \frac{1}{1-s}$ we get
		\[ x = \frac{1}{2}s(1-s)(1-2s), \]
	and
		\[ T_0^1(x,\emptyset) = \frac{1}{4}s^2(1-s)(1-3s). \]
	This implies that
		\[ T_0^3(x,\emptyset) = \frac{1}{2}s^3(1-s)(1-4s+2s^2). \]
	By a Theorem of Brown\cite{B65}, since $B(x,y)$ is a square it can be uniquely written as 
		\[ B(x,y) = Q(x,y)^2 R(x,y), \]
	where $Q(x,y)$ and $R(x,y)$ are polynomials in $y$ and where
		\[ R(x,0) = 1. \]
	We may suppose that $B(x,y) = (1 + yQ_1)^2(1 + R_1 y + R_2 y^2)$ where $Q_1, R_1, R_2$ are rational series in $s$. Solving this gives
		\[ Q_1 = s-1, \qquad R_1 = -2s, \qquad \mbox{ and } R_2 = -2s + 3s^2. \]
	It follows by the same argument as in \cite{G91} that $T_g(x,y,I)$ is a polynomial in $R^{\frac{1}{2}}(x,y)$, $R^{\frac{1}{2}}(x,z_i), i \in I$ and $R^{\frac{1}{2}}\left(x,\frac{1}{1-s}\right)$
	with coefficients which are rational series in $x, y, z_i, i \in I, R_1, R_2$ and $\frac{1}{1-s}$. More specifically, if $T_g(x) = [y^3] T_g(x,y,\emptyset)$ then
	$T_g(x)$ is a polynomial in $R^{\frac{1}{2}}\left( x, \frac{1}{1-s} \right)$ with coefficients which are rational series in $s$. Since 
	$R^{\frac{1}{2}}\left( x, \frac{1}{1-s} \right) = \frac{\eta^{\frac{1}{2}}}{1-s}$ it follows that $T_g(x)$ is a polynomial in $\eta^{\frac{1}{2}}$ with coefficients
	which are rational series in $s$. The result then follows from the fact that $\Lt_g(x) = x^{2g-2}T_g(x)$.
\end{proof}

We will now make use of Theorem~\ref{AlgebraicSeriesTheorem} and Theorem~\ref{MasterEquationGenus} to prove a stronger structure theorem for $\Lt_g(z)$. In addition, we will
show that the recurrence implied by Theorem~\ref{MasterEquationGenus} can be used to determine $\Lt_g(z)$ inductively.
It will be convenient for what follows to define a basis with which to work. For $i \geq 0$, define
	\[ \psi_i = \begin{cases} \frac{(1-2s)}{\eta^{\frac{i}{2}+1}}, & \mbox{ if } i \mbox{ is odd}, \\ \frac{1}{\eta^{\frac{i}{2}+1}}, & \mbox{ if } i \mbox{ is even}. \end{cases} \]
\begin{pro}\label{PsiBasis}
	The basis $\psi_i$ defined above satisfies the multiplicative identity,
		\[ \psi_i \psi_j = \chi(i,j) \psi_{i+j+2} + (1 - \chi(i,j))\psi_{i+j}, \]
	where
		\[ \chi(i,j) = \begin{cases} 1, & \mbox{ if either } i \mbox{ or } j \mbox{ is even}, \\ \frac{1}{3}, & \mbox{ if both } i \mbox{ and } j \mbox{ are odd}. \end{cases} \]
	Also, the action of the operator $D$ on the $\psi_i$ basis is given by
		\[ D\psi_i = \begin{cases} (i+2)\psi_{i+4} + (i+2)\psi_{i+2} - 2(i+2)\psi_i, & \mbox{ if } i \mbox{ is even}, \\ (i+2)\psi_{i+4} + i\psi_{i+2} - 2(i+1)\psi_i, & \mbox{ if } i \mbox{ is odd}. \end{cases} \]
\end{pro}
\begin{proof}
	Each of these identites follow from the definition of $\psi_i$, $\eta$, $s$ and $D$.
\end{proof}

\begin{thm}\label{StructureTheorem}
	For $g \geq 2$ the generating series for locally orientable cubic maps is of the form
		\[ \Lt_g = \sum_{i = 0}^{5g - 8} \mu_g(i) \psi_i. \]
	Furthermore, $\Lt_g$ for $g \geq 2$ can be determined recursively using the equation given in Corollary~\ref{MasterEquationGenusRefined}.
\end{thm}
\begin{proof}
	Using Theorem~\ref{GaoTheorem}, Corollary~\ref{MasterEquationGenusRefined} can be rearranged so that when written in the $\psi_i$ basis we have
	\[
		\left\{ \frac{4}{3} D^2 \right. \left. - 4(g\psi_2 + \psi_0 - 4)D - \frac{4}{3}(2\psi_4 + 9g(3-g)\psi_2 + 12\psi_0 - 32) \right\} \Lt_g = R_g(s),
	\]
	where $R_g(s)$ depends only on $\Lt_i$, $i = 0 \cdots g-1$. That $\Lt_g = \sum_{i=0}^{5g-8} \mu_g(i) \psi_i$ then follows by induction by directly computing
	the base case
		\[ \Lt_2 = \frac{23}{12} \psi_0 - 3 \psi_1 + \frac{13}{12} \psi_2, \]
	and then showing that
		\[ R_g(s) = \sum_{i=0}^{5g} r_g(i) \psi_i, \]
	for $g \geq 3$. Proving that $R_g(s)$ has this structure is a lengthy yet straightforward computation. One writes each of the parts in Corollary~\ref{MasterEquationGenusRefined}
	in terms of the $\psi_i$ basis and then using the identities in Proposition~\ref{PsiBasis} shows that $R_g(s)$ has an expansion in the $\psi_i$ basis and that the appropriate
	degree bounds hold. The result then follows using Theorem~\ref{AlgebraicSeriesTheorem}, expanding by partial fractions and then checking coefficients.
\end{proof}

\section{Asymptotic Behaviour}

In this final section we will consider some of the implications of the results above. In particular, we examine the leading coefficient in the $\psi$ basis
expansion of $\Lt_g(z)$ for each $g$. Using Theorem~\ref{MasterEquationGenusRefined} we can determine a recursion which the leading coefficients satisfy. Since the
asymptotic behaviour of each element in the $\psi$ basis can be determined, this allows us to determine the asymptotic behaviour of the series $\Lt_g(z)$ itself.
In what follows we will let \[ \al_g = \mu_g(5g - 8), \] i.e., $\al_g$ is the leading coefficient of $\Lt_g$ in the $\psi$ basis. Let
	\[ \be_k = \begin{cases} \frac{(5k-6)}{\sqrt{3}} \al_g, & \mbox{ if } k \mbox{ is odd}, \\ (5k-6) \al_g, & \mbox{ if } k \mbox{ is even}.\end{cases} \]

\begin{thm}
	For all $g > 1$,
		\begin{align*}
			\frac{8}{3} (g-1) \be_g	&= \frac{1}{162}(5g-8)(5g-12)(5g-6)\be_{g-2} \\
									&\qquad - \frac{1}{3^6 5!} (5g-6)(5g-10)(5g-14)(5g-18)(5g-22)\be_{g-4} \\
								    &\qquad + \frac{8(5g-6)}{432} \sum_{i=1}^{g-1} \be_i \be_{g-i} - \frac{(5g-6)}{2^3 3^6} \sum_{i=2}^{g-1} (5i-8)(5i-12) \be_{i-2} \be_{g-i} \\
								    &\qquad - \frac{1}{2^3 3^5} \sum_{i=1}^{g-1} \sum_{j=1}^{g-i-1} (5i-2) \be_i \be_j \be_{g-i-j}. 
		\end{align*}
\end{thm}
\begin{proof}
	Using the notation in the proof of Theorem~\ref{StructureTheorem}, we have
		\[ p_g(D) \Lt_g = \sum_{i=0}^{5g} r_g(i) \psi_i, \]
	where
		\[ p_g(D) = \frac{4}{3} D^2 - 4(g\psi_2 + \psi_0 - 4)D - \frac{4}{3}(2\psi_4 + 9g(3-g) \psi_2 + 12\psi_0 - 32). \]
	Now,
		\begin{align*}
			[\psi_{5g}] p_g(D) \Lt_g	
										&= \frac{4}{3} (5g-2)(5g-6) \al_g - 4g(5g-6)\al_g \\
										&= \frac{8}{3}(5g-6)(g-1)\al_g.
		\end{align*}
	So,
		\[ \frac{8}{3}(5g-6)(g-1)\al_g = r_g(5g). \]
	As in the proof of Theorem~\ref{StructureTheorem}, if we write out each term in Corollary~\ref{MasterEquationGenusRefined} in terms of the $\psi_i$
	basis and then simplify using the identities in Proposition~\ref{PsiBasis}, then after a very lengthy computation we can show that if we let $w_k = (5k-6)\al_k$ then
		\begin{align*}
			\frac{8}{3}(g-1)w_g	&= r_g(5g), \\
								&= \frac{1}{162} (5g-8)(5g-12)(5g-6)w_{g-2} \\
								& \qquad - \frac{1}{3^6 5!} (5g-6)(5g-10)(5g-14)(5g-18)(5g-22)w_{g-4} \\
								& \qquad + \frac{(5g-6)}{54} \sum_{i=1}^{g-1} \chi(i,g-i) w_i w_{g-i} \\
								& \qquad - \frac{(5g-6)}{2^3 3^6} \sum_{i=2}^{g-1} \chi(i,g-i) (5i-8)(5i-12) w_{i-2} w_{g-i} \\
								& \qquad - \frac{1}{9 \cdot 216} \sum_{i=1}^{g-2} \sum_{j=1}^{g-i-1} (5i-2)\chi(i,g-i)\chi(j,g-i-j)w_i w_j w_{g-i-j},
		\end{align*}
	where here
		\[ \chi(i,j) = \begin{cases} 1, & \mbox{ if either } i \mbox{ or } j \mbox{ is even}, \\ \frac{1}{3}, & \mbox{ if both } i \mbox{ and } j \mbox{ are odd}. \end{cases} \]
	If we make the substitution
		\[ w_k = 	\begin{cases}
						\sqrt{3} \beta_k, & \mbox{ if } k \mbox{ is odd}, \\
						\beta_k, & \mbox{ if } k \mbox{ is even},
					\end{cases} \]
	then it is straightforward to check that if $g$ is even, $\chi(i,g-i)w_i w_{g-i} = \be_i \be_{g-i}$ and $\chi(i,g-i)\chi(j,g-i-j)w_i w_j w_{g-i-j} = be_i \be_j \be_{g-i-j}$ for any $i$ and $j$ and that
	if $g$ is odd then $\chi(i,g-i)w_i w_{g-i} = \sqrt{3} \be_i \be_{g-i}$ and $\chi(i,g-i)\chi(j,g-i-j) w_i w_j w_{g-i-j} = \be_i \be_j \be_{g-i-j}$. So, taking the recursion for $w_g$ above and making the
	substitution, we get (after dividing by $\sqrt{3}$ if $g$ is odd) the desired result.
\end{proof}

Let
	\[ \be(z) = \sum_{n \geq 0} \be_n \left( \frac{3}{2} \right)^n z^{-\left( \frac{5n-2}{4} \right)}, \]
with $\be_0 = -36$ and $\be_1 = \frac{18}{\sqrt{3}}$.

\begin{cor} \label{AsymptoticDE}
	The generating series $\be(z)$ satisfies the differential equation
		\[ - \frac{8}{15} z \be'(z) - \frac{16}{15}\be(z) + \frac{8}{135}\be'''''(z) + \frac{2}{81}\be''(z)\be'(z) + \frac{2}{81}\be'''(z)\be(z) + \frac{1}{486} \be(z)^2\be'(z) = 0. \]
\end{cor}
\begin{proof}
	That the differential equation above is equivalent to the recursion for the coefficients $\be_k$ is easily checked by extracting coefficients.
\end{proof}

From the structure theorem for $\Lt_g$, we know that
	\[ \Lt_g \approx \begin{cases} \frac{\al_g}{\eta^{\frac{5g-6}{2}}}, & \mbox{ if } g \mbox{ is even}, \\ \frac{\al_g(1-2s)}{\eta^{\frac{5g-6}{2}}}, & \mbox{ if } g \mbox{ is odd}. \end{cases} \]
Further, it is a result of Gao\cite{G91} that
	\[ \eta \approx \left( \frac{\sqrt{6}}{2} \right)^{\frac{5g-6}{2}} \left( 1 - 12\sqrt{3}z \right)^{-\frac{(5g-6)}{4}}. \]
Darboux's theorem then implies that
	\[ \ell_g(n) = [z^n]\Lt_g \sim \frac{\be_g}{5g-6} \left( \frac{3}{2} \right)^{\frac{5g-6}{4}} \frac{n^{\frac{5(g-2)}{4}}}{\Gamma\left( \frac{5g-6}{4} \right)}\left( 12\sqrt{3} \right)^n. \]
In terms of the orientable map asymptotics constant $t_g$ and the nonorientable map asymptotics constant $p_g$, it is straightforward to show that
	\[ \be_k = \left( \frac{2}{3} \right)^k 9 \left( (5k-6)v_{k-1} - 4u_{k/2} \right), \]
where
	\begin{align*}
		u_n	&= 2^{n-2} \Gamma\left( \frac{5n-1}{4} \right)t_n, \\
		v_n &= 2^{\frac{n-3}{2}} \Gamma\left( \frac{5n-1}{4} \right) p_{\frac{n+1}{2}},
	\end{align*}
and we adopt the convention that $u_n = 0$ if $n$ is not an integer. In particular, $u_0 = 1$ and $v_0 = -\sqrt{3}$

Let
	\begin{align*}
		u(z)	&= z^{\frac{1}{2}} \sum_{n \geq 0} u_n z^{-\frac{5n}{2}}, \\
		v(z)	&= z^{\frac{1}{4}} \sum_{n \geq 0} v_n z^{-\frac{5n}{2}}.
	\end{align*}
Bender, Richmond and Gao\cite{BRG08} showed that $u(z)$ satisfies the ordinary differential equation
	\[ u^2 - \frac{1}{6} u'' = z. \]
Using the differential equation for $u$ and Corollary~\ref{AsymptoticDE} we can now prove Theorem~\ref{p_gTheorem} (Conjecture 1 of Garoufalidis and Mari\~no\cite{GM10}).
\begin{proof}[Proof of Theorem~\ref{p_gTheorem}]
	Let $v_n$ be the unique sequence of numbers such that the generating series \[v(z) = z^{1/4} \sum_{n \geq 0} v_n z^{-5n/4}\] satisfies the differential equation $2v' - v^2 + 3u = 0$ and with $v_0 = -\sqrt{3}$.
	Let \[ \beta(z) = -36(u(z) + v'(z)).\] Substituting this $\beta$ into the differential equation in Corollary~\ref{AsymptoticDE} and reducing using the relations $u^2 - \frac{1}{6} u'' = z$ and
	$2v' - v^2 + 3u = 0$ we find that $\beta$ is a solution.
	
	Also, checking $\be_0$ and $\be_1$ we see that this is the unique solution which determines the map asymptotics constants.
\end{proof}

\bibliographystyle{plain}
\bibliography{document}
	
\end{document}